\documentclass[11pt]{article}
\usepackage{amsfonts}
\usepackage{amssymb}
\usepackage{latexsym}
\usepackage{amsmath}
\usepackage{amsthm}
\usepackage{a4wide}
\usepackage{epsfig}
\usepackage{mathtools}
\usepackage[pagebackref,hyperindex]{hyperref}
\usepackage{color}

\theoremstyle{plain}

\newtheorem{theorem}{Theorem}[section]

\newtheorem{lemma}[theorem]{Lemma}
\newtheorem{corollary}[theorem]{Corollary}
\newtheorem{definition}[theorem]{Definition}

\newcommand{\toto}{\leftrightarrow}
\newcommand{\To}{\Rightarrow}

\let\Pi\undefined
\DeclareMathSymbol{\Pi}{\mathalpha}{operators}{"05}

\newcommand{\alg}[1]{{\ensuremath{\boldsymbol #1}}}

\newcommand{\Sent}{\mbox{\rm Sent}_\Gamma}

\newcommand{\MTL}{{\ensuremath{\mathrm{MTL}}}}

\newcommand{\MV}{{\ensuremath{\mathrm{MV}}}}

\newcommand{\model}[1]{{\mathbf{{#1}}}}

\newcommand{\K}{{\mathbb{K}}}

\begin{document}

\title{
\bf\Large Solution of some problems in \\the arithmetical complexity of first-order fuzzy logics}

\author{
  \addtocounter{footnote}{1}
  F\'elix Bou \footnote{
  The authors acknowledge partial support of Eurocores (LOMOREVI
  Eurocores Project FP006/FFI2008-03126-E/FILO), Spanish Ministry of
  Education and Science (project TASSAT TIN2010-20967-C04-01),
  Catalan Government (2009SGR-1433/4), and the FP7-PEOPLE-2009-IRSES project MaToMUVI (PIRSES-GA-2009- 247584). Carles Noguera also acknowledges
  support from the research contract ``Juan de la Cierva''  JCI-2009-05453.
  }\\
  University of Barcelona\\
  bou@ub.edu
  \and
  Carles Noguera $^\fnsymbol{footnote}$\\
  IIIA -- CSIC\\
  cnoguera@iiia.csic.es
}


\date{}
\maketitle

\begin{abstract}
  This short paper addresses the open problems left in the
  paper~\cite{MoNo10}. Besides giving solutions to these two problems,
  some clarification concerning the role of the full vocabulary
  (including functional symbols) in the proofs there given is also
  discussed.
  
  \medskip

\end{abstract}

The paper~\cite{MoNo10} obtains general results providing positions in
the arithmetical hierarchy for first-order fuzzy logics.
It also formulates two problems that are left without solution:


\begin{description}
  \item \textsc{Open Problem 1} (see~\cite[p.~408]{MoNo10}): 
    ``Show that for every class $\K$ of chains, the set ${\rm
    SAT}_\mathrm{pos}(\K)$ is $\Pi_1$-hard.''

  \item \textsc{Open Problem 2} (see~\cite[p.~421]{MoNo10}):
    ``Is it true that ${\rm finTAUT}_\mathrm{pos}({\L\forall}) \subseteq {\rm
    stTAUT}_\mathrm{pos}({\L\forall})$? This would imply: ${\rm
    stTAUT}_\mathrm{pos}({\L\forall})  = {\rm canratTAUT}_\mathrm{pos}({\L\forall}) =
    {\rm finTAUT}_\mathrm{pos}({\L\forall})$ and ${\rm
    stSAT}_{1}({\L\forall})  = {\rm canratSAT}_{1}({\L\forall}) =
    {\rm finSAT}_{1}({\L\forall})$.''
\end{description}
The aim of this short paper is to provide answers for these two open
problems. We point out that the answers (including the proofs) given
here will be included in the forthcoming \cite{HaMoNo11}, which, among
other stuff, provides an updated version of the kind of results
studied in \cite{MoNo10}.

\paragraph{Notation and Background.}
The notation used in this paper corresponds to the one introduced in
\cite{MoNo10}, so we advise the reader to get acquaintance with the
terminology there used before reading this paper. 
It is worth noticing two facts concerning these issues. The first one is
that although in that paper the authors consider both the first-order
language with and without $\Delta$, since here our concern only involves
the previous two open problems we can always assume that we are in the
first-order language without $\Delta$. 
It should also be emphasized that \cite{MoNo10} 
only considered the \emph{full vocabulary}, i.e., the first-order vocabulary which includes
a countable number of constant symbols, a countable number of predicate
symbols and a countable number of functional symbols. In the classical
setting it is obvious that there is no distinction, from an expressive
power point of view, between considering the full vocabulary or the full
predicate one (i.e., the language with a
countable number of constant symbols and a countable number of predicate
symbols, with no functional symbols); but this is not at all obvious
in the fuzzy setting. Thus, the design choice of the full vocabulary
in~\cite{MoNo10} is, for the sake of generality, a drawback.

\paragraph{Structure of the paper.}
The first open problem is answered positively in Section~\ref{sec:problem1}.
The proof here given follows the same idea than the one given for proving
$\Sigma_1$-hardness of positive tautologies
(see~\cite[Theorem~3.15]{MoNo10}). The only difference is that the role
there played by the algebraic term $2x$ (there defined as
$\neg (\neg x \& \neg x)$) is replaced here by the much more common term
$x^2$ (as usual defined by $x \& x$); so in some sense this proof can be
considered simpler than the one given for positive tautologies in
\cite{MoNo10}.
In the last part of this section we point out that the proof here given,
and also the proof given in \cite[Theorem~3.15]{MoNo10}, rely on the
crucial fact that the vocabulary considered is the full one (including
functional symbols). Thus, although the statement
\begin{quotation}
  ``This theorem, in particular, solves a couple of open problems recently
  proposed by H\'ajek in \cite{Ha09b}; namely given a set $\K$ of standard
  BL-chains such that its corresponding logic ${\rm L}_{\K}\forall$ is
  recursively axiomatizable show that ${\rm genTAUT}_{1} ({\rm
  L}_{\K}\forall)$ and ${\rm genTAUT}_\mathrm{pos} ({\rm
  L}_{\K}\forall)$ are $\Sigma_1$-hard.''
  \hfill \cite[p.~409]{MoNo10}
\end{quotation}
is right,\footnote{It is worth emphasizing that this quotation refers to
the full vocabulary
} H\'ajek's problem remains open for the full predicate
vocabulary (and also for other vocabularies).
\medskip

The second open problem is considered in Section~\ref{sec:problem2}.
This problem is narrower than the other in the sense that it only
involves the {\L}ukasiewicz case. In this section the authors notice
that the second open problem was indeed answered negatively by
\textsc{P. H\'ajek} in \cite[Lemma~4]{Ha02a}.


\section{Positive Satisfiability is $\Pi_1$-hard}
\label{sec:problem1}

In this section we answer the first open problem positively. 
For the rest of the section, let us fix $\K$ a non-trivial class of
\MTL-chains, i.e., $\K$ contains some \MTL-chain with at least two
elements. Let us remind that the set ${\rm SAT}_\mathrm{pos}(\K)$ of positive
satisfiable (first-order) sentences is defined as
\begin{center} ${\rm SAT}_\mathrm{pos}(\mathbb{K}) \coloneqq \{\varphi \in
  \Sent \mid$ there exist $\alg{A} \in \mathbb{K}$ and an
  $\alg{A}$-structure $\model{M}$ such that
  $\|\varphi\|^\alg{A}_\model{M} > \overline{0}^\alg{A}\}$.
\end{center} 
For the purpose of this section we next introduce two auxiliary sets of
sentences:
\begin{itemize}
  \item ${\rm TAUT}_0(\mathbb{K}) \coloneqq \{\varphi \in \Sent \mid$
    for every $\alg{A} \in \mathbb{K}$ and every $\alg{A}$-structure
    $\model{M}$, $\|\varphi\|^\alg{A}_\model{M} =
    \overline{0}^\alg{A}\}$.

  \item ${\rm TAUT}_{<1}(\mathbb{K}) \coloneqq \{\varphi \in \Sent \mid$
    for every $\alg{A} \in \mathbb{K}$ and every $\alg{A}$-structure
    $\model{M}$, $\|\varphi\|^\alg{A}_\model{M} <
    \overline{1}^\alg{A}\}$.
\end{itemize}

The following step in our proof is the following lemmata and their
consequences. We remind again the reader that this proof
is very close to the one given in \cite[Theorem~3.15]{MoNo10}.

\begin{lemma}
  \label{lem:equation}
  The equation $x^2 \land (\neg x)^2 \approx \overline{0}$ holds in all \MTL-algebras.
\end{lemma}

\begin{proof}
  It is obvious that \MTL-chains satisfy that for every element $a$,
  it holds that
    $a^2 \land (\neg a)^2 = (a \land \neg a)^2 \leq a \ast \neg a = 0$.
\end{proof}

\begin{corollary}
  Let $\K$ be a class of \MTL-chains. Then, for every
  sentence $\varphi$, it holds that $\varphi^2 \land (\neg \varphi)^2
  \in {\rm TAUT}_0(\mathbb{K})$.
\end{corollary}




\begin{definition}
  For every formula $\varphi$ (in the classical setting) we define the
  formula $\varphi^{\star}$ (in the fuzzy setting) through the following
  clauses:
  \begin{itemize}
    \item if $\varphi$ is a literal (i.e., either an atomic formula or the
      negation of an atomic formula), then $\varphi^{\star}\coloneqq
      \varphi^2$ (i.e., $\varphi \& \varphi$).
    \item $(\varphi_1 \land \varphi_2)^{\star}\coloneqq
      \varphi_1^{\star} \land \varphi_2^{\star}$,
    \item $(\varphi_1 \lor \varphi_2)^{\star}\coloneqq
      \varphi_1^{\star} \lor \varphi_2^{\star}$,
    \item $(\forall x \varphi)^{\star}\coloneqq \forall x
      (\varphi^{\star})$,
    \item $(\exists x \varphi)^{\star}\coloneqq \exists x
      (\varphi^{\star})$.
  \end{itemize} 
\end{definition}

It is obvious from the previous definition that $\varphi$ and
$\varphi^{\star}$ always have same free variables. In particular,  if
$\varphi$ is a sentence, then $\varphi^*$ is also a sentence.

\begin{lemma}
  \label{l:propositional-lemma}
  Let $\K$ be a non-trivial class of \MTL-chains, and let $\varphi$ be a
  lattice combination of literals.  The following are equivalent:
  \begin{itemize}
    \item[(1)] $\varphi$ is a classical propositional contradiction.
    \item[(2)] $\varphi^\star \in {\rm TAUT}_0(\mathbb{K})$.
  \end{itemize}
\end{lemma}

\begin{proof}
First of all we show $(1) \To (2)$. By distributivity, $\varphi$ can be
equivalently written as $\bigvee_{i=1}^n \bigwedge_{j=1}^{n_i}
\alpha_{i,j}$, where $\alpha_{i,j}$ are literals. Thus, $\varphi$ is a
classical contradiction iff for every $i \in \{1, \ldots, n\}$,
$\bigwedge_{j=1}^{n_i} \alpha_{i,j}$ is a classical contradiction.
Therefore, for every $i \in \{1, \ldots, n\}$ there are $j_1,j_2 \in
\{1, \ldots, n_i\}$ such that $\alpha_{i,j_1} = \neg \alpha_{i,j_2}$.
Hence, $\alpha_{i,j_1}^2 \land \alpha_{i,j_2}^2$ belongs to ${\rm
TAUT}_0(\mathbb{K})$ by Lemma~\ref{lem:equation}. Since this formula is
implied by $\bigwedge_{j=1}^{n_i}\alpha_{i,j}^2$, we have that
$\varphi^\star$ also belongs to ${\rm TAUT}_0(\mathbb{K})$.

$(2) \To (1)$ can be easily proved by contraposition. If $\varphi$ is
not a classical propositional contradiction, then there is an evaluation
$e$ on $\alg{B}_2$ such that $e(\varphi) = 1$. Since $\varphi^\star$
and $\varphi$ are equivalent in classical logic, we also have
$e(\varphi^\star) = 1$. Now, given any $\alg{A} \in \K$, it is clear
that $e$ can also be seen as an evaluation on $\alg{A}$ and
$e(\varphi^\star) = \overline{1}$.
\end{proof}

\begin{lemma}[Dual Herbrand's Theorem]
A purely universal sentence $\forall x_1 \ldots \forall x_n \ \psi(x_1,
\ldots, x_n)$ is a classical contradiction if, and only if, there exists
$m$ and closed terms $\{t_1^i, \ldots, t_n^i \mid i = 1, \ldots m\}$
such that $\bigwedge_{i=1}^m \psi(t_1^i, \ldots, t_n^i)$ is a classical
propositional contradiction.
\end{lemma}


\begin{proof}
  We notice that each one of the following statements is equivalent to
  the others.
  \begin{enumerate}
    \item $\forall x_1 \ldots \forall x_n \ \psi(x_1, \ldots, x_n)$ is a
      classical contradiction. 
    \item $\neg \forall x_1 \ldots \forall x_n \ \psi(x_1, \ldots, x_n)$ is a
      classical tautology.
    \item $\exists x_1 \ldots \exists x_n \ \neg\psi(x_1, \ldots, x_n)$ is a
      classical tautology.
    \item There are closed terms $\{t_1^i, \ldots, t_n^i \mid i = 1,
      \ldots m\}$ such that $\bigvee_{i=1}^m \neg \psi(t_1^i, \ldots,
      t_n^i)$ is a classical propositional tautology. 
    \item There are closed terms $\{t_1^i, \ldots, t_n^i \mid i = 1,
      \ldots m\}$ such that $\bigwedge_{i=1}^m \psi(t_1^i, \ldots,
      t_n^i)$ is a classical propositional contradiction.
  \end{enumerate}
  The only non trivial step is the one between 3 and 4, and this
  one is obtained by Herbrand's Theorem.
\end{proof}



\begin{lemma}
  Let $\K$ be a non-trivial class of \MTL-chains, and let $\varphi$ be
  $\forall x_1 \ldots \forall x_n \ \psi(x_1, \ldots, x_n)$ where $\psi$
  is a lattice combination of literals. The following are equivalent:
  \begin{itemize}
    \item[(1)] $\varphi \in {\rm TAUT}_0(\alg{B}_2)$.
    \item[(2)] $\varphi^\star \in {\rm TAUT}_0(\K)$.
    \item[(3)] $\varphi^\star \in {\rm TAUT}_{< 1}(\K)$.
  \end{itemize}
\end{lemma}

\begin{proof}
The only non trivial implication is $(1) \To (2)$. Suppose that
$\varphi$ is a classical contradiction. By the dual Herbrand's Theorem,
there are closed terms $t_j^i$ such that $\bigwedge_{i=1}^m \psi(t_1^i,
\ldots, t_n^i)$ is a classical propositional contradiction. By
Lemma~\ref{l:propositional-lemma}, recalling that $^\star$ commutes with
$\land$, we have that $\bigwedge_{i=1}^m \psi^\star(t_1^i, \ldots,
t_n^i) \in {\rm TAUT}_0(\K)$. Therefore, $\varphi^\star =
\forall x_1 \ldots \forall x_n \ \psi^\star(x_1, \ldots, x_n) \in {\rm
TAUT}_0(\K)$.
\end{proof}

\begin{lemma}
The set of classical purely universal first-order contradictions is
$\Sigma_1$-hard.
\end{lemma}

\begin{proof}
First observe that the set all contradictions is $\Sigma_1$-hard.
Indeed, the set of all tautologies is $\Pi_1$-hard and we have that for
any sentence $\varphi$, $\varphi$ is a contradiction iff $\neg \varphi$
is a tautology.  Now given any sentence $\varphi$ we can write the
following chain of equivalencies: $\varphi$ is a contradiction iff $\neg
\varphi$ is a tautology iff its Herbrand form (purely existential)
$(\neg \varphi)^H$ is a tautology iff $\neg (\neg \varphi)^H$ is a
contradiction. The latter is a purely universal form, so we are done.
\end{proof}

\begin{theorem}
  \label{t:TAUTareSigma1hard}
  Let $\K$ be a non-trivial class of \MTL-chains. The set ${\rm
  TAUT}_0(\K)$ is $\Sigma_1$-hard and thus ${\rm SAT}_\mathrm{pos}(\K)$ is
  $\Pi_1$-hard.
\end{theorem}

\begin{proof}
It follows from the previous two lemmata and the fact that ${\rm
SAT}_\mathrm{pos}(\K)$ is the complementary set of ${\rm TAUT}_0(\K)$.
\end{proof}

To finish this section we point out that this proof, and the same for
the proof in \cite[Theorem~3.15]{MoNo10}, does not work for the full
predicate vocabulary. The reason is that in this vocabulary the set of
purely universal contradictions is indeed decidable. This is a
particular case of the decidability of the satisfiability problem (in
the classical setting) for relational (i.e., without functional symbols)
$\exists^* \forall^*$-sentences without equality. This decidability
result was proved long ago by Bernays and Sch\"onfinkel (the reader
interested on this topic can find more details in
\cite[Section~6.2.2]{BoGrGu97}).

\section{The Second Open Problem}
\label{sec:problem2}


Rutledge proved in \cite{Ru60} that the set of $1$-tautologies
over the standard \MV-chain coincide with the intersection of the sets of
$1$-tautologies over finite \MV-chains.

The fact that concerning $1$-satisfiability we can distinguish the
standard \MV-chain from the finite ones, can be obtained using the
sentence
\[
\Phi \coloneqq \quad  \exists x (P(x)\toto \neg P(x)) \: \& \: \forall x \exists y (P(x)\leftrightarrow
(P(y)\& P(y))),
\]
which was already considered in \cite[Lemma~4]{Ha02a}. It is quite
simple to check that $\Phi$ is $1$-satisfiable in some structure over
the standard \MV-chain, while it cannot be $1$-satisfiable in structures
over a finite \MV-chain. In other words, $\Phi \in {\rm
stSAT}_{1}({\L\forall})$ while $\Phi \not \in {\rm
finSAT}_{1}({\L\forall})$. An immediate corollary of this fact is that
$\neg \Phi \in {\rm finTAUT}_\mathrm{pos}({\L\forall})$, while $\neg \Phi \not
\in {\rm stTAUT}_\mathrm{pos}({\L\forall})$, which settles negatively the
second open problem stated above.

\end{document}